\documentclass[11pt]{article}
\usepackage{amsmath,amsfonts,amssymb,amsthm}
\usepackage{epsfig}
\usepackage{mathrsfs}
\usepackage{enumerate}
\newtheorem{theorem}{Theorem}
\newtheorem{lemma}[theorem]{Lemma}
\newtheorem{koro}[theorem]{Corollary}
\newtheorem{rem}[theorem]{Remark}
\usepackage[utf8]{inputenc}

\newcommand{\fed}{\,\rule{.1mm}{.24cm}\rule{.24cm}{.1mm}\,}

\newcommand{\relint}{{\mathrm{relint}}}

\DeclareMathOperator{\Nor}{{\mathrm{Nor}}}
\newcommand{\HH}{\mathcal{H}}

\newcommand{\htau}{\tau}
\newcommand{\R}{{\mathbb R}}

\newcommand{\N}{{\mathbb N}}

\newcommand{\Z}{{\mathbb Z}}

\newcommand{\K}{{\mathcal K}}

\allowdisplaybreaks

\usepackage{geometry} 
\usepackage{graphicx} 

\usepackage{verbatim} 

\usepackage{hyperref}
\usepackage{times}

\catcode`@=11
\def\section{%
\setcounter{equation}{0} \setcounter{theorem}{0} \@startsection
{section}{1}{\z@}{-4.0ex plus -1ex minus
    -.2ex}{2.3ex plus .2ex}{\bf}}
\def\subsection{\@startsection{subsection}{2}{\z@}{-3.25ex plus-1ex
    minus-.2ex}{1.5ex plus.2ex}{\reset@font\bf}}

\begin{document}

\title{Extensions of translation invariant valuations on polytopes}

\author{Wolfram Hinderer, Daniel Hug and 
Wolfgang Weil}
\date{\today}

\maketitle
\begin{abstract}
\noindent We study translation invariant, real-valued valuations on the class of convex polytopes in Euclidean space and discuss which continuity properties are sufficient for an extension of such valuations to all convex bodies. For this purpose, we introduce flag support measures of convex bodies via a local Steiner formula and derive 
some of the properties of these measures. 

\medskip

\noindent\textit{Key words:}\medskip
Valuation, Steiner formula, support measure, flag measure

\end{abstract}

\section{Translation invariant valuations}

Let $\K$ be the space of  convex bodies (non-empty compact convex sets) in $\R^d$, $d\ge 3$, with scalar product $\langle\cdot\, , \cdot\rangle$ and norm $\|\cdot\|$. We  endow $\K$ with the topology derived from the Hausdorff metric $d_H$ (see \cite{S} for background information and notions not explicitly defined  here). Let $\cal P\subset\K$ be the set of convex polytopes in $\R^d$. A real- or measure-valued functional $\varphi$ on $\cal K$ or $\cal P$ is called a {\it valuation}, if it is additive in the sense that
$$
\varphi (K\cup M) + \varphi (K\cap M) = \varphi (K) + \varphi (M) ,
$$
whenever $K,M$ and $K\cup M$ lie in ${\cal K}$ (resp. in ${\cal P}$). For a real-valued valuation $\varphi$ on ${\cal P}$ which is invariant under translations and weakly continuous (that is, continuous with respect to parallel displacements of the facets of the polytopes), McMullen \cite{McM83,McM93} has shown that
\begin{equation}\label{val}
\varphi (P) =\varphi_0(P)+\sum_{j=1}^{d-1} \sum_{F\in {\cal F}_j(P)} f_{j}(n(P,F)) V_j(F) + \varphi_d(P) ,
\end{equation}
where $\varphi_0(P)=\varphi(\{0\})$ (cf.~\cite[pp.~214-5]{McMS} or \cite[Theorem 6.4.7]{S}) and $\varphi_d$ is 
a multiple of the volume functional. For $j\in\{0, \ldots,d-1\}$, ${\cal F}_j(P)$  is the collection of $j$-faces of $P$,  $V_j$ is the $j$-th intrinsic volume, $n(P,F)$ is the intersection of the normal cone $N(P,F)$ of $P$ at $F$ and the unit sphere $S^{d-1}$, which is a $(d-j-1)$-dimensional spherical polytope, and $f_{j}$ is a simple, additive functional on the class ${\wp}_{d-j-1}$ of  at most $(d-j-1)$-dimensional spherical polytopes. 
As usual,  $f_j$ is said to be 
simple if it is zero on spherical polytopes of dimension less than $d-j-1$.  In comparison to related work (see \cite{McM83,McM93}), we prefer to express our formulas in terms of spherically convex subsets of the unit sphere instead of using the convex cones spanned by such subsets. 

Conversely, any sequence of simple,  additive functionals $f_j : {\wp}_{d-1-j} \to \R$, $j=1,\ldots,d-1$,  yields a weakly continuous, translation invariant valuation $\varphi$ on ${\cal P}$ via \eqref{val}. To be more precise,
\begin{equation}\label{phij}
\varphi_j (P) =  \sum_{F\in {\cal F}_j(P)} f_{j}(n(P,F)) V_j(F)
\end{equation}
defines a weakly continuous, translation invariant valuation $\varphi_j$ on ${\cal P}$ which is homogeneous of degree $j$, for $j=1,\ldots,d-1$. 
Thus, \eqref{val} corresponds to the decomposition $\varphi = \sum_{j=0}^d \varphi_j$ of $\varphi$ into homogeneous components (see \cite{McM83,McM93}). We call $f_{1},\ldots,f_{d-1}$  {\it associated functions} of $\varphi$. If $P$ has dimension $k<j$, then ${\cal F}_j(P) = \emptyset$ and so $\varphi_j(P)=0$, by definition. 

Convex polytopes are dense in $\K$. Therefore it is  natural  to ask which $j$-homogeneous valuations $\varphi_j$, $j\in\{1,\ldots,d-1\}$, on ${\cal P}$ allow an extension to ${\cal K}$ by approximation, in a continuous way. In a first attempt, one could ask whether a valuation on ${\cal P}$ which is continuous with respect to the Hausdorff metric can be extended continuously to ${\cal K}$. This seems to be a difficult problem and neither a positive result nor a counterexample is known. A second, also quite natural approach, would be to ask for appropriate continuity conditions on the associated function $f_j$ which guarantee a continuous extension (and thus also imply that $\varphi_j$ is continuous on ${\cal P}$). Throughout this paper, we follow this second line.

A rather strong sufficient condition would be to assume that
\begin{equation}\label{cont1}
f_j (p) = \int_p \tilde f_j(u) \, \HH^{d-j-1}(du),\quad p\in {\wp}_{d-j-1},
\end{equation}
for some continuous function $\tilde f_j$ on $S^{d-1}$. Here we write $\HH^s$, $s\ge 0$, for the $s$-dimensional Hausdorff measure in Euclidean space $\R^d$. The $j$-th area measure $S_j(P,\cdot)$, $j\in\{0,\ldots,d-1\}$, of a polytope $P$ is a Borel measure on the unit sphere $S^{d-1}$  (see \cite[Section 4.2]{S} for an introduction to area measures), which is given by
$$
S_j(P,\cdot) = \binom{d-1}{j}^{-1} \sum_{F\in {\cal F}_j(P)} V_j(F) {\int_{n(P,F)}}
 {\bf 1}\{ u\in \cdot\}\, \HH^{d-j-1}(du).
$$
Hence, assuming \eqref{cont1}, it  follows that
$$
\varphi_j (P) = \binom{d-1}{j} \int_{S^{d-1}} \tilde f_j(u) \,S_j(P,du).
$$
The map $K\mapsto S_j(K,\cdot)$ is weakly continuous, that is, continuous with respect to the Hausdorff metric $d_H$ on $\K$ and the weak 
topology on the space of finite Borel measures on $S^{d-1}$. Therefore
$$
\varphi_j (K) = \binom{d-1}{j} \int_{S^{d-1}} \tilde f_j(u)\,S_j(K,du),
$$
for $K\in\K$, 
defines a continuous extension of $\varphi_j$ to ${\cal K}$.  
We remark that \eqref{cont1} 
is fulfilled if  $\varphi_j$ is  smooth in the sense of Alesker \cite{A2, Al} (see also \cite{Bernig2011} 
for a recent survey on valuations and applications to integral geometry). 
 
A much weaker natural condition would be to require that $f_j$ is continuous with respect to the Hausdorff metric on spherical polytopes.
In the following, we shall see that even a weak absolute continuity of $f_j$, which implies continuity of $f_j$, is not sufficient. On the other hand, we shall show that a continuous extension exists, if $f_j$ satisfies a special kind of absolute continuity, which is still weaker than \eqref{cont1}.

Subsequently, we consider weakly continuous, $j$-homogeneous, translation invariant valuations $\varphi_j$ on ${\cal P}$, for $j\in\{1,\ldots,d-1\}$. We first specify the different continuity properties for such valuations, which we shall study. We call $\varphi_j$ {\it strongly continuous} if it has an associated function $f_j$ which satisfies \eqref{cont1} with a continuous function $\tilde f_j$ on $S^{d-1}$. Further, we say that $\varphi_j$ is {\it flag-continuous}, if it has an associated function $f_j$   such that \eqref{cont1} holds with a continuous function $\tilde f_j$, which may depend on both, the normal vector $u$ and the linear subspace $\langle p\rangle$ generated by $p$. To be more precise, we introduce, for $q\in\{1,\ldots, d-1\}$, the {\it flag manifold}
$$
F(d,q) := \{ (u,L)\in S^{d-1}\times G(d,q): u\in L\} ,
$$
where $G(d,q)$ is the Grassmannian of $q$-dimensional linear subspaces of $\R^d$. 
With the usual topology, $F(d,q)$ is a compact space. 
The property of flag-continuity for $\varphi_j$ means that
\begin{equation}\label{cont2}
f_j (p) = \int_p \tilde f_j(u,\langle p\rangle)\, \HH^{d-j-1}(du),\quad p\in {\wp}_{d-j-1},
\end{equation}
for some continuous function $\tilde f_j$ on $F(d,d-j)$, if $\dim(\langle p\rangle) =d-1-j$, and $f_j(p)=0$  if $\dim(\langle p\rangle) <d-1-j$.  We call a flag-continuous valuation $\varphi_j$ {\it strongly flag-continuous}, if the function $\tilde f_j$ on $F(d,d-j)$ arises as the image $\tilde f_j = T_{j} \tilde g_j$ of some continuous function $\tilde g_j$ on $F(d,d-j)$ under the transform 
$$
T_{j} : C(F(d,d-j))\to C(F(d,d-j))
$$ given by
\begin{equation}\label{integraltransform}
(T_{j}\tilde h) (u,L) = \int_{G(\langle u\rangle ,d-j)} [M,L^\perp]^2\, \tilde h(u,M) \,\nu^{\langle u\rangle}_{d-j}(dM) ,
\end{equation}
for $(u,L)\in F(d,d-j)$. 
For subspaces $U_i\in G(d,q_i)$, $i=1,2$, with $q:=q_1+q_2\le d$, the expression $[U_1,U_2]$ is defined as in 
\cite[p.~597-8]{SW}. In particular, if $a_1,\ldots,a_{q_1}$ is an orthonormal basis of $U_1$ and $b_1,\ldots,b_{q_2}$ is an orthonormal basis of $U_2$, then $[U_1,U_2]$ is the $q$-dimensional volume of the parallelepiped spanned by $a_1,\ldots,a_{q_1}, b_1,\ldots,b_{q_2}$. Hence, $[U_1,U_2]=|\det(a_1,\ldots,a_{q_1}, b_1,\ldots,b_{q_2})|$ if ${\rm dim}(U_1+U_2)=q$, 
where the determinant is calculated in $U_1+U_2$, 
 and $[U_1,U_2]=0$ if ${\rm dim}(U_1+U_2)<q$.  The Grassmannian $G(\langle u\rangle ,d-j)$ consists of all $U\in G(d,d-j)$ which contain the linear subspace $\langle u\rangle$ spanned by $u$. Moreover,  $\nu^{\langle u\rangle}_{d-j}$ is the unique $SO(u)$-invariant  probability measure on $G(\langle u\rangle ,d-j)$, where $SO(u)$ denotes the set of all  proper (orientation preserving) rotations of $\R^d$ which fix  $u\in S^{d-1}$. 
Finally, we observe that if $\varphi_j$ is flag-continuous, then the function $f_j$ in \eqref{cont2} is continuous on  ${\wp}_{d-j-1}$ with respect to the Hausdorff metric. 

From these definitions, it follows immediately that the following implications hold:
\begin{align*}
 \varphi_j&{\rm \ strongly\ continuous\ }\Rightarrow \varphi_j{\rm \ strongly\ flag-continuous\ }\\ & \Rightarrow \varphi_j{\rm \ flag-continuous\ }\Rightarrow \varphi_j{\rm \ weakly\ continuous\ }.
\end{align*}

Since for $j=d-1$ the three notions, strong continuity, strong flag-continuity and flag-continuity  all coincide, and since then a continuous extension to all convex bodies is always possible,  we assume that $j\in\{1,\ldots,d-2\}$ throughout the following. 

In the next section, we shall show that, for $j\in\{1,\ldots,d-2\}$, flag-continuity is not sufficient for the existence of a continuous extension of $\varphi_j$ to ${\cal K}$.  The counterexample we present is based on a certain flag measure for polytopes and on the construction of two sequences of convex polytopes $(P_k)_{k\in\N}$ and $(Q_k)_{k\in\N}$, which converge to the same convex body, but for which a suitably constructed flag-continuous valuation $\varphi_j$ has different limits, that is, $\lim_{k\to\infty}\varphi_j(P_k)\neq \lim_{k\to\infty}\varphi_j(Q_k)$. However, in Section 4 we shall see that an extension exists, if $\varphi_j$ is strongly flag-continuous. This result makes use of another sequence of flag measures which we introduce and study in Section 3. These flag measures are obtained as coefficient measures of a local Steiner formula, first for convex polytopes and then for general convex bodies by an approximation argument. 
This approach extends the one well known for support measures of convex bodies. 

The question of extendability of valuations from polytopes to arbitrary bodies is also discussed in a recent preprint by S. Alesker \cite{A3}. For $1\le j\le d-2$, he shows the existence of a smooth function $\tilde f_j$ such that the valuation given by \eqref{phij}, with $f_j$ and $\tilde f_j$ related by 
\eqref{cont2},  has no continuous extension to ${\cal K}$. His approach uses sophisticated techniques from representation theory, whereas our counterexample is based on a specific geometric construction. In \cite{A3}, also the vector space of functions $\tilde f_j$ which allow a continuous extension is described for $j=1$ and dimension $d=3$, again in an abstract setting. Our results in Section 4 give sufficient conditions for extendability, for all values of $d$ and $j$. We are grateful to S. Alesker for useful discussions on our two different approaches.

\section{Valuations without continuous extensions}

In order to show that flag-continuity is not sufficient for the existence of a continuous extension of $\varphi_j$ to ${\cal K}$, for an arbitrary polytope $P$ in $\R^d$, we define  a flag measure  $\htau_j(P,\cdot)$, for $j\in\{ 0,\ldots,d-1\}$,  on $F(d,d-j)$ by
\begin{equation}\label{firstflagmeasure}
\htau_j(P,\cdot ) := \sum_{F\in{\cal F}_j(P)} V_j(F)\, \int_{n(P,F)} {\bf 1}\{ (u,F^\bot)\in \cdot\}\, \HH^{d-j-1}(du) ,
\end{equation}
where $F^\perp\in G(d,d-j)$ denotes the subspace  orthogonal to the linear subspace $L(F)\in G(d,j)$  parallel to $F$. 
For $d=3$ and $j=1$ this measure was introduced by Ambartzumian \cite{Amb1, Amb2}, who used it for an integral representation of the width function and a subsequent characterization of zonotopes.

It follows from \eqref{firstflagmeasure} that 
$$
\htau_j(P,F(d,d-j))= \sum_{F\in{\cal F}_j(P)} V_j(F) \,\HH^{d-j-1}(n(P,F))=\omega_{d-j}V_j(P),
$$
where $\omega_n:=2\pi^{n/2}/\Gamma(n/2)$ is the $(n-1)$-dimensional Hausdorff measure of the unit sphere $S^{n-1}$.  
Hence, for a sequence $(P_k)$, $k\in\N$, with $P_k\to K\in\K$, as $k\to\infty$, the measures $\htau_j(P_k,\cdot )$ are uniformly bounded, for $k\in\N$, since $\htau_j(P_k,F(d,d-j))=\omega_{d-j}V_j(P_k)\to \omega_{d-j}V_j(K)$, as $k\to\infty$, by the continuity of the intrinsic volumes. Hence, there is a subsequence $\htau_j(P_{k_i},\cdot )$, $i\in\N$, which converges weakly to a measure $\htau$ on $F(d,d-j)$. The following result  shows that, for $j\in\{ 1,\ldots,d-2\}$, the limit measure $\htau$ in general will depend on the approximating sequence $P_k\to K$, hence there is no continuous extension of the measures $\htau_j(P,\cdot )$ from polytopes $P$ to arbitrary convex bodies $K\in{\cal K}$. 

We add a few comments on the construction of the polytopes $P_k$, which are used in the proof, as well as on the 
strategy of the proof. In \cite{Hind}, in dimension $d=3$ a polytope $P_k\subset\R^d$, $k\in\N$, is defined as  the convex hull of those 
points on the unit sphere whose projection to $\R^{d-1}$ lies in the discrete set $2^{-k}\Z^{d-1}$. The polytopes $P_k$ 
converge to the rotation invariant unit ball $B^d$, as $k\to\infty$, and a subsequence of $\htau_j(P_k,\cdot)$ converges weakly to a 
measure $\htau$ on $F(d,d-j)$. If $\htau$ is not rotation invariant, then there is a rotation $\vartheta\in SO_d$ such that 
$\vartheta \htau\neq\htau$, and hence there is a continuous function $\tilde f$ on $F(d,d-j)$ which separates $\vartheta\htau$ and $\htau$. 
Using $\tilde f$, we can define a flag-continuous valuation $\varphi$ such that $
\lim_{k\to\infty} \varphi (P_k)\not= \lim_{k\to\infty} \varphi (\vartheta P_k)$. In order to show that $\htau$ is not rotation invariant, it is 
crucial to know and control the set of $j$-faces of $P_k$. Already for $d=3$ and $j=1$ (which is considered in \cite{Hind}) this turns out 
to be a delicate task and at least one symmetry argument in \cite{Hind} is apparently not available. 

For this reason, we use a modified construction for a sequence of polytopes that was first described, and used for a different purpose, in 
\cite{HS}. In \cite{HS},  a polytope $P_t$, for $t>0$, is defined as the convex hull of a subset of those points on a rotational paraboloid whose projection to $\R^{d-1}$ lies in the discrete set $2t\Z^{d-1}$. For these polytopes, the set of $j$-faces, for general dimension and $j\in\{0,\ldots,d-1\}$, is explicitly determined in \cite{HS}. The limit set $K$  of $(P_t)$, as $t\to 0^+$, (that is, part of the paraboloid)  is then  invariant under rotations fixing the vertical axis. In contrast, the facial structure of the polytopes 
$P_t$, for $t>0$, remains rigid in the sense that the orthogonal projections to $\R^{d-1}$ of the linear subspaces which are parallel to the faces of the polytopes $P_t$ are coordinate subspaces. 
This is enough to admit a modification of the argument described in the preceding paragraph (and based on \cite{Hind}). The construction of the polytopes $P_t$ is related to a well known construction in computational geometry used to determine the Delaunay triangulation of a point set in $\R^{d-1}$ via the convex hull of the set of points vertically lifted to a paraboloid (see \cite[Section 13.1/2, Observation 13.13]{Edelsbrunner} and \cite[Section 7]{JT2013}).

\begin{theorem} For each $j\in\{1,\ldots, d-2\}$, there is a flag-continuous $j$-homogeneous valuation $\varphi= \varphi_j$ on ${\cal P}$ and a convex body $K\in\K$ such that there exist two sequences of polytopes, $(P_k)$ and $(Q_k)$, $k\in\N$, with $P_k \to K$ and $ Q_k\to K$, as $k\to\infty$, and such that the limits $\lim_{k\to\infty} \varphi (P_k)$ and $\lim_{k\to\infty} \varphi (Q_k)$ exist but satisfy
$$
\lim_{k\to\infty} \varphi (P_k)\not= \lim_{k\to\infty} \varphi (Q_k).
$$
\end{theorem}

\begin{proof} Let $j\in\{1,\ldots,d-2\}$ be fixed. Let $e_1,\ldots,e_d$ be the standard (orthonormal) basis of $\R^d$. We identify $\R^{d-1}$ and $\R^{d-1}\times\{0\}\subset\R^d$. For $t>0$ and $e=e_1+\ldots+e_{d-1}$, the cubes $t(e+2z)+[-t,t]^{d-1}$, $z\in \Z^{d-1}$, are the cells of a polytopal complex in $\R^{d-1}$, denoted by $\mathcal{C}_t$, with vertex set $2t\Z^{d-1}$. Let  $l_-:\R^{d-1}\to\R^d$ and $l_+:\R^{d-1}\to\R^d$ be the functions defined by
$$
\ell_-(x):=-e_d+x+\|x\|^2e_d,\quad \ell_+(x):=e_d+x-\|x\|^2e_d,\qquad x\in\R^{d-1}.
$$
Further, we define the two sets
$$
\mathcal{L}_\pm(t):=\ell_\pm(\{2tz:z\in \Z^{d-1},1-4t^2\|z\|^2\ge 0\}).
$$
Then 
$$
P_t:=\text{conv}(\mathcal{L}_-(t)\cup\mathcal{L}_+(t))
$$
is a convex polytope, with vertices on the rotational paraboloids $\ell_-(\R^{d-1})$ and $\ell_+(\R^{d-1})$, respectively. As $t\to 0^+$,  
the polytopes $P_t$ converge to the convex body
$$
K:=\{(x,s)\in\R^{d-1}\times\R:|s|\le 1-\|x\|^2\},
$$
which is symmetric with respect to $\R^{d-1}$ and invariant under rotations $\vartheta\in SO(e_d)$. 
By the remarks preceding the statement of the theorem, we may assume that $\htau_j(P_{t},\cdot )\to \htau$ weakly, as $t\to 0^+$. The rotation group acts in a natural way on $F(d,d-j)$, that is, for $(u,L)\in F(d,d-j)$ and $\vartheta\in SO_d$, we consider the operation $SO_d\times F(d,d-j)\to F(d,d-j)$ given by $\vartheta (u,L):=(\vartheta u,\vartheta L)$. 
We shall show that $\htau$ is not $SO(e_d)$-invariant, that is,  there exists a rotation $\vartheta\in SO(e_d)$ such that the image measure $\vartheta \htau$ of $\htau$ under $\vartheta$ is different from $\htau$. If this is shown, we put $Q_t:= \vartheta P_t$, which implies that $Q_t\to K$, as $t\to 0^+$. Moreover, we have
\begin{align*}
\htau_j(Q_t,\cdot)&=\htau_j(\vartheta P_t,\cdot)\\
&= \sum_{F\in{\cal F}_j(\vartheta P_t)} V_j(F)\int_{n(\vartheta P_t,F)}\mathbf{1}\{(u,F^\perp)\in\cdot\} \,\HH^{d-j-1}(du)\\
&= \sum_{G\in{\cal F}_j( P_t)} V_j(\vartheta G)\int_{n(\vartheta P_t,\vartheta G)}\mathbf{1}\{(u,(\vartheta G)^\perp)\in\cdot\} \,\HH^{d-j-1}(du)\\
&= \sum_{G\in{\cal F}_j( P_t)} V_j( G)\int_{n( P_t, G)}\mathbf{1}\{\vartheta (u,  G^\perp)\in\cdot\} \,\HH^{d-j-1}(du)\\
&=\htau_j(P_t,\vartheta^{-1}(\cdot))=\vartheta\htau_j(P_t,\cdot).
\end{align*}
Hence, $\htau_j(Q_{t},\cdot )= \vartheta \htau_j(P_{t},\cdot ) \to \vartheta  \htau$ as $t\to 0^+$. Since $\vartheta \htau\neq\htau$, there exists a continuous  function (in fact, a  function of class $C^\infty$) $\tilde f$ on $F(d,d-j)$ such that
\begin{equation}\label{limit}
\int_{F(d,d-j)} \tilde f(u,L)\ (\vartheta \htau) (d(u,L) ) \not= \int_{F(d,d-j)} \tilde f(u,L)\ \htau (d(u,L) ).
\end{equation}
A flag-continuous, $j$-homogeneous valuation $\varphi$ on ${\cal P}$ is defined by
\begin{align*}
\varphi(P):= & \sum_{F\in{\cal F}_j(P)}V_j(F)\int_{n(P,F)} \tilde f(u,F^\bot)\ \HH^{d-1-j}(du)\\
=& \int_{F(d,d-j)}\tilde f(u,L)\, \tau_j(P,d(u,L)).
\end{align*}
Since $\tau_j(P_t,\cdot)\to \tau$ weakly, as $t\to 0^+$, it follows that $\varphi$ satisfies
$$
 \varphi (P_t)
 \to \int_{F(d,d-j)} \tilde f (u,L)\ \htau (d(u,L) ) 
$$
and 
\begin{align*}
 \varphi (Q_t)&=\int_{F(d,d-j)}\tilde f (u,L)\ \vartheta\htau_j(P_t,d(u,L) )\\
 &\to \int_{F(d,d-j)} \tilde f (u,L)\ (\vartheta \htau) (d(u,L) )  ,
\end{align*}
as $t\to 0^+$. 
Thus, \eqref{limit} implies that
$$
\lim_{t\to 0^+} \varphi (P_t)\not= \lim_{t\to 0^+} \varphi (Q_t).
$$

It remains to be shown that $\vartheta \htau \not= \htau$, for a suitable rotation $\vartheta\in SO(e_d)$. 
It is sufficient to show that the measure $\tilde\tau$ on $G(d,j)$, which is the image measure of $\htau $ under the mapping $F(d,d-j)\to G(d,j)$,  $(u,L)\mapsto L^\bot$, is not $SO(e_d)$-invariant. The latter follows, if we find a Borel set $A\subset G(d,j)$ with $\tilde\htau (A)>0$ and such that $\tilde\nu_j (A)=0$, for each $SO(e_d)$-invariant finite measure $\tilde\nu_j$ on  $G(d,j)$. 
 
We consider the finitely many subspaces parallel to the $j$-faces of the cube $[-1,1]^{d-1}$, that is, 
$$
\mathcal{D}:=\{L(G) : G\in\mathcal{F}_j([-1,1]^{d-1})\}.
$$
Each $j$-face $G$ of one of the cubes of $\mathcal{C}_t$ satisfies $L( G)\in \mathcal{D}$. 
For $L\in G(d,j)$, we denote by $L\vert e_d^\perp$ the orthogonal projection of $L$ to $e_d^\perp$ and define 
$$
A:= \{ L\in G(d,j):\|e_d\vert L\|\le \sqrt{2}^{-1},L\vert e_d^\perp\in\mathcal{D}\}.
$$
It is easy to see that $A$ is a closed subset of $G(d,j)$. 
Let $\tilde\nu_j$ denote an arbitrary finite $SO(e_d)$-invariant measure on $G(d,j)$. We show that $\tilde\nu_j(A)=0$. To see this, let $\nu^{e_d}$ denote the 
unique Haar probability measure on $SO(e_d)$. Then, by the assumed invariance we have
\begin{align*}
\tilde\nu_j(A)&=\int_{G(d,j)}\int_{SO(e_d)}\mathbf{1}\{\vartheta L\in A\}\, \nu^{e_d}(d\vartheta)\, \tilde\nu_j(dL)\\
&\le\int_{G(d,j)}\int_{SO(e_d)} \mathbf{1}\{(\vartheta L)\vert e_d^\perp\in \mathcal{D}\}\, \nu^{e_d}(d\vartheta)\, \tilde\nu_j(dL)\\
&=\int_{G(d,j)}\int_{SO(e_d)} \mathbf{1}\{\vartheta (L\vert e_d^\perp)\in \mathcal{D}\}\, \nu^{e_d}(d\vartheta)\,\tilde\nu_j(dL).
\end{align*}
Since $j\in\{1,\ldots,d-2\}$ and $\mathcal{D}$ is finite, we clearly have
$$
\int_{SO(e_d)} \mathbf{1}\{\vartheta (L\vert e_d^\perp)\in \mathcal{D}\}\, \nu^{e_d}(d\vartheta)=0,
$$
which implies the required assertion. 

Thus it remains to be shown that $\tilde\htau (A)>0$. 

Let $A_t$, for $t\in (0,1/8)$, be the set of $j$-dimensional subspaces parallel to some $j$-face of $P_t$ having non-empty intersection with  $B_0:=\{x\in\R^{d-1}:\|x\|\le 1/4\}\times [0,\infty)$. Since the orthogonal projection to $\R^{d-1}$ of each such $j$-face of $P_t$ is a 
$j$-face of one of the cubes of $\mathcal{C}_t$ and each such $j$-face of $P_t$ is obtained by vertically lifting a $j$-face of one of the cubes of $\mathcal{C}_t$ (see \cite{HS} 
for a detailed argument and \cite[Proposition 7.17]{JT2013} for the main fact on which the argument is based), it follows that $A_t\subset A$ once we have shown that $\|e_d\vert L\|\le \sqrt{2}^{-1}$ for all $L\in A_t$. 
In fact, if $u\in L\cap S^{d-1}$ and $L\in A_t$, then there are $x_1,x_2\in \R^{d-1}$ with $x_1\neq x_2$, $\|x_i\|\le 1/2$, for $i=1,2$, and such that 
$$
u=\frac{e_d+x_2-\|x_2\|^2e_d-[e_d+x_1-\|x_1\|^2e_d]}{\|e_d+x_2-\|x_2\|^2e_d-[e_d+x_1-\|x_1\|^2e_d]\|}.
$$
Then we get
\begin{equation}\label{eqa}
|\langle e_d,u\rangle|\le \frac{|\|x_2\|^2-\|x_1\|^2|}{\sqrt{(\|x_2\|^2-\|x_1\|^2)^2+\|x_2-x_1\|^2}}.
\end{equation}
Since $\|x_i\|\le 1/2$, $i=1,2$, we have $\|x_1\|+\|x_2\|\le 1$, and thus
\begin{equation}\label{eqb}
(\|x_2\|^2-\|x_1\|^2)^2\le (\|x_2\|-\|x_1\|)^2\le \|x_2-x_1\|^2.
\end{equation}
From \eqref{eqa} and \eqref{eqb}, we deduce that
$$
|\langle e_d,u\rangle|\le \frac{|\|x_2\|^2-\|x_1\|^2|}{\sqrt{2(\|x_2\|^2-\|x_1\|^2)^2}} = \frac{1}{\sqrt{2}},
$$
which yields the required assertion.

Then we obtain  
\begin{align*}
&\htau_j(P_t,\{(u,L^\perp)\in S^{d-1}\times G(d,d-j):L\in A\})\\
&\quad \ge \htau_j(P_t,\{(u,L^\perp)\in S^{d-1}\times G(d,d-j):L\in A_t\})\\
&\quad= \sum_{F\in \mathcal{F}_j(P_t)}V_j(F)\int_{n(P_t,F)}\mathbf{1}\{L(F)\in A_t\}\,\HH^{d-1-j}(du)\\
&\quad\ge  \sum_{F\in \mathcal{F}_j(P_t)}\mathcal{H}^j(F\cap B_0)\HH^{d-1-j}(n(P_t,F))\\
&\quad= \binom{d-1}{j} C_j(P_t,B_0)\to \binom{d-1}{j} C_j(K,B_0)>0,
\end{align*}
as $t\to 0^+$. Here $C_j(M,\cdot)$ denotes the $j$-th curvature measure of a convex body $M$ in $\R^d$ (see \cite[Section 4.2]{S} or  \cite[(14.13)]{SW}). 
Note that the convergence holds, since the boundary of the half cylinder $B_0$ has $C_j(K,\cdot )$-measure 0. Since 
$\{(u,L^\perp)\in S^{d-1}\times G(d,d-j):L\in A\}$ is a closed subset of $S^{d-1}\times G(d,d-j)$ and $\htau_j(P_{t},\cdot )\to \htau$ weakly, as $t\to 0^+$, we obtain that
\begin{align*}
\tilde\htau(A)&=\htau(\{(u,L^\perp)\in S^{d-1}\times G(d,d-j):L\in A\}) \\
&\ge \limsup_{t\to 0^+}\htau_j(P_t,\{(u,L^\perp)\in S^{d-1}\times G(d,d-j):L\in A\})\\
&\ge \binom{d-1}{j} C_j(K,B_0)>0,
\end{align*}
which completes the argument.

The required sequences are then obtained by choosing $t=1/k$, $k\in\N$.
\end{proof}

\section{Flag measures for convex bodies}

We now introduce various sequences of natural flag measures for arbitrary convex bodies $K\in {\cal K}$. 
For $j\in\{0,\ldots,d-1\}$, these measures will be defined on the flag manifold $F(d,d-j)$, on the manifold
$$
F^\bot (d,j):=\{ (u,L) \in S^{d-1}\times G(d,j) : L\perp u\} ,
$$
or on $N(d,j) := \R^d\times F^\bot(d,j)\subset \R^d\times S^{d-1}\times G(d,j)$. In the latter case, 
 the measures will be concentrated on a generalization of the normal bundle $\Nor K$ of $K$, that is, on the set
$$
\Nor_j (K) := \{(x,u,L) \in N(d,j) : (x,u)\in \Nor K  \} .
$$
Recall that if $h(K,u):=\max\{\langle x,u\rangle:x\in K\}$ denotes the support function $h(K,\cdot)$ of $K$ evaluated at $u\in\R^d$, 
then $\Nor K:=\{(x,u)\in K\times S^{d-1}:\langle x,u\rangle=h(K,u)\}$ is the normal bundle of $K$. 

We start with polytopes and give a direct definition of a flag measure in the spirit of \eqref{firstflagmeasure}. 
Namely, let $\psi_j(P,\cdot)$ be the measure on $F(d,d-j)$ given by
\begin{align}\label{secondflagmeasure}
\psi_j(P,\cdot ) :=& \sum_{F\in{\cal F}_j(P)} V_j(F) \int_{n(P,F)} 
\int_{G(\langle u\rangle, d-j)} [L,F]^2\,{\bf 1}\{ (u,L)\in \cdot\} 
\notag\\
&\qquad  \ \nu_{d-j}^{\langle u\rangle} (dL) \,\HH^{d-j-1}(du) ,
\end{align}
where we simply write $[L,F]$ instead of $[L,L(F)]$. 
Moreover, here and in the following, for a given subspace $U\in G(d,l)$, we write $G(U,k)$ for the set of all $L\in G(d,k)$ with $U\subset L$ if $l\le k$, 
and with $L\subset U$ if $k\le l$. The corresponding Haar probability measure on $G(U,k)$ is denoted by $\nu_k^U$. (If $U=\R^d$, then the upper index is omitted.) 
A comparison of \eqref{secondflagmeasure} with \eqref{firstflagmeasure} shows that  
$$
\psi_j(P,\cdot ) = T_{j} \htau_j(P,\cdot ),
$$
that is, $\psi_j(P,\cdot )$ is the image measure 
of $\htau_j(P,\cdot )$ under the integral transform $T_{j}$ given by \eqref{integraltransform} (since $T_j$ is apparently self-adjoint, it can be extended to a linear transform on measures, by duality).

As we shall show in this section, the measure $\psi_j(P,\cdot )$ has a continuous extension to all convex bodies. This also implies that the transform $T_{j}$ is not injective. Namely, for the sequences $P_k\to K$, $Q_k\to K$, $k\in\N$, considered in Theorem 2.1, we saw in the proof that $\htau_j (P_k,\cdot) \to \htau$, $\htau_j (Q_k,\cdot ) \to \htau'$ with $\htau\not=\htau'$. Since $T_{j} : C(F(d,d-j))\to C(F(d,d-j))$ is continuous, this implies 
$$\psi_j(P_k,\cdot ) = T_{j}  \htau_j(P_k,\cdot ) \to T_{j}  \htau$$
 and 
 $$\psi_j(Q_k,\cdot ) = T_{j}  \htau_j(Q_k,\cdot ) \to T_{j}  \htau',$$ 
 hence
$$
\psi_j(K,\cdot ) =  T_{j}  \htau =  T_{j} \htau' ,
$$
but $\htau\not = \htau'$.

In order to show that $\psi_j(P,\cdot )$ has a continuous extension to all convex bodies, we introduce a more general framework. We consider, for $k\in\{0,\ldots,d-1\}$, the affine Grassmannian $A(d,k)$ of $k$-flats in $\R^d$, together with the (suitably normalized) motion invariant measure $\mu_k$ on $A(d,k)$. Clearly, $A(d,k)$ can be represented as
\begin{equation}\label{repr}
A(d,k) = \{ \rho (L_0 + x) : x\in L_0^\perp, \rho \in SO_d\} ,
\end{equation}
where $L_0\in G(d,k)$ is a fixed subspace, and then $\mu_k$ can be defined by
\begin{align*}
\mu_k(\cdot) :=\ & \int_{SO_d}\int_{L_0^\perp} {\bf 1}\{ \rho (L_0+x)\in \cdot\} \,\HH^{d-k}(dx)\,\nu(d\rho ) \\
=\ &\int_{G(d,k)}\int_{L^\perp} {\bf 1}\{ L+x\in \cdot\} \,\HH^{d-k}(dx)\,\nu_k(dL ),
\end{align*}
where $\nu$ denotes the Haar probability measure on $SO_d$. 
We refer to \cite[Sections 13.1 and 13.2]{SW} for further details. 
For $K\in\K$ and $E\in A(d,k)$, we define the distance 
$$d(K,E):=\min\{\|x-y\|:x\in K,y\in E\}.
$$ 
If there is a unique pair of 
points $(x,y)\in K\times E$ with $d(K,E)=\|x-y\|$, then we put $p(K,E):=x$ and $l(K,E):=y$. In this case, $p(K,E)$ is the unique nearest 
point of $K$ to $E$, which is called the {\em metric projection} of $l(K,E)$ to $K$. Note that in these definitions we do not require that $K\cap E=\emptyset$. If $K\cap E=\emptyset$ and if the distance $d(K,E)>0$ is realized by a unique pair of points, then we define 
$$
u(K,E):=\frac{l(K,E)-p(K,E)}{d(K,E)}.
$$
(In fact, for the definition of $u(K,E)$ the uniqueness of the pair of points is not needed, any pair of distance minimizing points could be used.) Let $L=L(E)\in G(d,k)$  denote the linear subspace parallel to $E\in A(d,k)$. The set
$$
K^{(k)} := \{ E\in A(d,k) : E\cap K=\emptyset, p(K,E)\ {\rm exists}\}.
$$ 
is Borel measurable.  
This can be shown by the methods provided in \cite[Section 12 and 13.2]{SW}. 
In the following, all sets encountered are Borel measurable (which will be used without further mentioning it). 

\begin{lemma}\label{L21}
Let $K\in\K$ and $k\in\{0,\ldots,d-1\}$. Then $E\in K^{(k)}$ for $\mu_k$-almost all $E\in A(d,k)$ with $E\cap K=\emptyset$. 
Further, if $E\in K^{(k)}$, then  $\pi(E):=(p(K,E),u(K,E),L(E))\in \Nor_k(K)$, and the mapping $\pi$ is continuous on 
$K^{(k)}$.
\end{lemma}

\begin{proof} We first notice, that
\begin{equation}\label{nullset}
\mu_k(\{ E\in A(d,k) : E\cap K=\emptyset, E\notin K^{(k)}\} )=0 .
\end{equation}
In fact, any flat $E\in A(d,k)\setminus K^{(k)}$ with $E\cap K=\emptyset$ is parallel to some line segment in the boundary of $K$. Then, \eqref{nullset} follows from \cite[Corollary 2.3.11]{S}.

For $E\in  K^{(k)}$, the fact that $\pi(E)\in \Nor_k(K)$ is obvious. It remains to show that $\pi$ is continuous. We start with the continuity of the map $\K\times A(d,k)\to [0,\infty)$, $(K,E)\mapsto d(K,E)$. For this, let $K_i,K\in\K$, $i\in\N$, with $K_i\to K$ in the Hausdorff metric. Further, let $L_0\in G(d,k)$, $x_i,x_0\in L^\perp_0$ and $\rho_i,\rho_0\in SO_d$, for $i\in\N$, such that $\rho_i\to\rho_0$ and $x_i\to x_0$, as $i\to\infty$. By the choice of the topology on $A(d,k)$ 
it is sufficient to show that $d(K_i,\rho_i(L_0+x_i))\to d(K,\rho_0(L_0+x_0))$, as $i\to\infty$. But this follows from the estimate
\begin{align*}
\left|d(K_i,\rho_i(L_0+x_i))-d(K,\rho_0(L_0+x_0))\right|&=\left|d(\rho_i^{-1}K_i-x_i,L_0))-d(\rho_0^{-1}K-x_0,L_0)\right|\\
&\le {d}_H(\rho_i^{-1}K_i-x_i,\rho_0^{-1}K-x_0).
\end{align*}

Next we show that $p(K,\cdot)$ is continuous on $K^{(k)}$. Due to the compactness of $K$, it suffices to show that any accumulation point $x$ of a sequence $p(K,E_i)$, $i\in\N$, with $ E_i\to E$ and $E_i,E\in K^{(k)}$, coincides with $p(K,E)$. In fact, let $p(K,E_{i_j})\to x\in K$, for some subsequence $(E_{i_j})_{j\in\N}$. Since $d(p(K,E_{i_j}),E_{i_j}) \le d(p(K,E),E_{i_j})$ implies that $d(x,E)\le d(K,E)$, the uniqueness of $p(K,E)$ shows that $x=p(K,E)$.

The continuity of $l(K,\cdot)$ on $K^{(k)}$ follows in a similar way. Then this also implies the continuity of $u(K,\cdot)$.

The continuity of $E\mapsto L(E)$ is an easy consequence of the definition of the topologies on $A(d,k)$ and $G(d,k)$ (cf.~\cite[p.~582]{SW}).
\end{proof}

By similar arguments as in the proof of Lemma \ref{L21}, one can show the continuity of $p(\cdot ,E), l(\cdot,E)$ and $u(\cdot, E)$, if $E$ is an appropriate flat. We state the result without proof.

\begin{lemma}\label{L24}
Let $K_i,K\in\K$ be convex bodies with $K_i\to K$, as $i\to\infty$. Let
$$
E\in K^{(k)}\cap\bigcap_{i=1}^\infty K_i^{(k)} .
$$
Then $p(K_i,E)\to p(K,E)$, $l(K_i,E)\to l(K,E)$ and $u(K_i,E)\to u(K,E)$, as $i\to\infty$.
\end{lemma}

For $\epsilon > 0$, the set
$$
K^{(k)}_{\epsilon} := \{ E\in K^{(k)} : d(K,E)\le \epsilon\}
$$
is Borel measurable. Hence, if $\mu_{k}\fed K^{(k)}_{\epsilon}$ denotes the restriction of $\mu_k$ to $K^{(k)}_{\epsilon}$, then 
the image measure $\mu^{(k)}_{\epsilon}(K,\cdot)$ of 
$\mu_{k}\fed K^{(k)}_{\epsilon}$ under $\pi$ is a finite Borel measure on $N(d,k)\subset \R^d\times S^{d-1}\times G(d,k)$. If we define, for a Borel set $C$ in $N(d,k)$, 
$$
M^{(k)}_\epsilon(K,C) := \{ E\in K^{(k)}_\epsilon : \pi (E)\in C\},
$$
then 
$$
\mu^{(k)}_{\epsilon}(K,C)=\mu_k(M^{(k)}_\epsilon(K,C)).
$$

The following is our main result in this section.

\begin{theorem}\label{steiner} {\bf (a)} For each convex body $K$ in $\R^d$ and each $k\in\{ 0,\ldots,d-1\}$, there exist finite Borel measures $\Theta^{(k)}_0(K,\cdot), \ldots, \Theta^{(k)}_{d-k-1}(K,\cdot)$ on $N(d,k)$, concentrated on $\Nor_k(K)$, such that
\begin{equation}\label{localsteiner}
\mu^{(k)}_{\epsilon}(K,\cdot) = \frac{1}{d-k}\sum_{m=0}^{d-k-1}\epsilon^{d-k-m}\binom{d-k}{ m} \Theta^{(k)}_m(K,\cdot) ,
\end{equation}
for each $\epsilon >0$.

\medskip\noindent
{\bf (b)}
For each $m\in\{ 0,\ldots,d-k-1\}$, the mapping $K\mapsto \Theta^{(k)}_m(K,\cdot)$, from $\K$ to the space of finite 
Borel measures on $N(d,k)$, is weakly continuous and additive, i.e.
$$
\Theta^{(k)}_m(K\cup M,\cdot)+ \Theta^{(k)}_m(K\cap M,\cdot) = \Theta^{(k)}_m(K,\cdot)+\Theta^{(k)}_m(M,\cdot) ,$$
for all convex bodies $K,M$ for which $K\cup M$ is convex.

\medskip\noindent
{\bf (c)}
For each Borel set $B$ in $N(d,k)$, the mapping  $K\mapsto \Theta^{(k)}_m(K,B )$ is measurable.
\end{theorem}

We call $\Theta^{(k)}_m(K,\cdot)$ the {\em $m$-th $k$-flag support measure} of $K$. Before we start with the proof of Theorem \ref{steiner}, we collect a few remarks. The defining equation \eqref{localsteiner} for the measures $\Theta^{(k)}_m(K,\cdot)$ can be seen as a {\it local Steiner formula} in $A(d,k)$. The investigation of local Steiner formulas for convex bodies is a classical topic by now and the method we employ here has been used first to introduce the {\it curvature measures} and the {\it (surface) area measures} of convex bodies, as well as their common generalization, the {\it support measures}. We refer  to \cite[Chapter 4]{S}, for details. As we shall see later, these classical measures appear as projections of our general ones. Other projections on smaller flag manifolds have been considered before. In \cite{W1, W2}, corresponding {\it (generalized) curvature measures} on the flag manifold
$$
K(d,k) :=\{(x,E) : x\in K, E\in A(d,k) \ {\rm touches}\ K\ {\rm in}\ x\}
$$
were considered and in the diploma thesis \cite{Kr} analogous {\it (generalized) surface area measures} on $F(d,k)$ were studied. The unified approach presented here, leading to the {\it flag support measures} $\Theta^{(k)}_m(K,\cdot)$, follows along similar lines. 
It was already indicated in \cite{Kr} and, in more detail, in \cite{Hind}.   In \cite{HTW}, flag support measures $\Xi^{(k)}_m(K,\cdot)$ are discussed from a different point of view and with a different normalization. A comparison of Theorem \ref{steiner} and of \cite[Theorem 3]{HTW} shows that
\begin{equation}\label{stern}
\Xi^{(k)}_m(K,\cdot)=\frac{\binom{d-k-1}{m}}{\omega_{d-k-m}}\, \Theta^{(k)}_m(K,\cdot).
\end{equation} 
Moreover, in \cite{HTW} mixed flag measures and applications to zonoids are explored. 
A measure geometric approach to flag measures of convex bodies 
and functions is provided in \cite{BHW}.

\medskip

We divide the proof of Theorem \ref{steiner} into several lemmas.

\begin{lemma}\label{l1} For fixed $\epsilon >0$, the map $
K\mapsto \mu^{(k)}_{\epsilon}(K,\cdot)$ 
is weakly continuous and additive. Moreover, for every Borel set $B$ in $N(d,k)$ the 
map $K\mapsto \mu^{(k)}_{\epsilon}(K,B)$ is measurable.
\end{lemma}

\begin{proof} Let $K_i,K$ be convex bodies with $K_i\to K$, as $i\to\infty$, and let $A$ be the set of all $k$-flats that are parallel to a line segment (of positive length) in the boundary of some $K_i$, $i\in\N$. As mentioned in the proof of Lemma \ref{L21}, we have $\mu_k(A)=0$. Let $B\subset N(d,k)$ be open (with respect to the induced subspace topology) and let $E\in M^{(k)}_\epsilon(K,B)\setminus A$  be a flat with $0<d(K,E)<\epsilon$. Then, for almost all $i$, the sets $K_i$ and $E$ do not intersect. Moreover, $d(K_i,E)\to d(K,E)$ and, by Lemma \ref{L24},  $(p(K_i,E),u(K_i,E))\to (p(K,E),u(K,E))$, as $i\to\infty$. It follows that, for almost all $i$,  $d(K_i,E)<\epsilon$ and $(p(K_i,E),u(K_i,E),L(E))\in B$. Thus, for almost all $i$, we have $E\in M^{(k)}_\epsilon(K_i,B)$. 
Hence,  we get
$$
\left(M^{(k)}_\epsilon(K,B)\setminus A\right) \cap \{E\in A(d,k) : d(K,E)<\epsilon\}\subset \liminf_{i\to\infty} M^{(k)}_\epsilon(K_i,B),
$$
and thus
\begin{align*}
\mu^{(k)}_\epsilon (K,B )&= \mu_k(M^{(k)}_\epsilon(K,B))\cr
&= \mu_k(M^{(k)}_\epsilon(K,B) \cap \{E\in A(d,k) : d(K,E)<\epsilon\})\cr
&\le \mu_k\left(\liminf_{i\to\infty} M^{(k)}_\epsilon(K_i,B)\right)\cr
&\le \liminf_{i\to\infty} \mu_k(M^{(k)}_\epsilon(K_i,B))\cr
&= \liminf_{i\to\infty} \mu^{(k)}_\epsilon(K_i,B) .
\end{align*}
Here, we have used Fatou's lemma and, in the second equation, the fact that
$$
\mu_k(\{ E\in G(d,k) : d(K,E)=\epsilon\}) =0,
$$
which follows, for instance, from Crofton's formula (see \cite{SW}).

By the same kind of arguments, we get that
$$
\mu_\epsilon (K_i,N(d,k)) \to \mu_\epsilon (K,N(d,k)),
$$
as $i\to\infty$, which completes the proof of the weak continuity.

Concerning the additivity, we follow the corresponding arguments for support measures (see, e.g., \cite{SW}). Let $K,M$ and $K\cup M$ be  convex bodies, and let $E\in K^{(k)}\cap M^{(k)}$. We put $y=p(K,E), z=p(M,E)$ and consider, first, the case $p(K\cup M)=y$.

As $K\cup M$ is convex, the segment $[y,z]$ lies in $K\cup M$. Therefore, there exists a point $x\in [y,z]\cap K\cap M$. The mapping $t\mapsto d(tz+(1-t)y,E)$ 
is convex on $[0,1]$ and has a minimum at $t=0$. (Here we identify a point $x$ with the set $\{x\}$.) This implies $d(y,E)\le d(x,E)\le d(z,E)$. Since $z=p(M,E)$, we have $d(x,E)\ge d(z,E)$ and thus $d(x,E)=d(z,E)$ (and $x,z\in M$). The uniqueness of the nearest point now implies $z=x\in K\cap M$. We get
$$
d(K\cup M,E)=d(K,E),\quad d(K\cap M,E)= d(M,E)
$$
and 
$$
u(K\cup M,E)=u(K,E),\quad u(K\cap M,E)= u(M,E).
$$
Let $C\subset N(d,k)$ be a Borel set. We obtain that the statements
$$
E\in M^{(k)}_\epsilon(K\cup M,C) \quad {\rm and}\quad E\in M^{(k)}_\epsilon(K,C)
$$
are equivalent, as are the statements
$$
E\in M^{(k)}_\epsilon(K\cap M,C)\quad {\rm and}\quad E\in  M^{(k)}_\epsilon(M,C).
$$

In the other case, $p(K\cup M,E) = z$, we get in a similar way that
$$
E\in  M^{(k)}_\epsilon(K\cup M,C) \quad {\rm and}\quad E\in M^{(k)}_\epsilon(M,C)
$$
are equivalent, as well as
$$
E\in M^{(k)}_\epsilon(K\cap M,C) \quad {\rm and}\quad E\in M^{(k)}_\epsilon(K,C).
$$

This means that, for $\mu_k$-almost all $E$, we have the identity
\begin{align*}
&{\bf 1}\{E\in M^{(k)}_\epsilon(K\cup M,C)\} + {\bf 1}\{E\in M^{(k)}_\epsilon(K\cap M,C)\}\cr
&\hspace{3cm} = {\bf 1}\{E\in M^{(k)}_\epsilon(K,C)\} + {\bf 1}\{E\in M^{(k)}_\epsilon(M,C)\} .
\end{align*}
Integration with respect to $\mu_k$ yields the additivity property.

The remaining assertion is implied by \cite[Lemma 12.1.1]{SW}.
\end{proof}

\begin{lemma}\label{l2} Let $P$ be a polytope. Then there exist finite Borel measures $\Theta^{(k)}_j(P,\cdot)$  on $N(d,k)$, for $j=0,\ldots,d-k-1$, such that
\begin{equation}\label{localsteiner2}
\mu^{(k)}_{\epsilon}(P,\cdot) = \frac{1}{d-k}\sum_{m=0}^{d-k-1}\epsilon^{d-k-m}\binom{d-k}{ m} \Theta^{(k)}_m(P,\cdot) ,
\end{equation}
for each $\epsilon >0$.
\end{lemma}

\begin{proof} For each flat $E\in P^{(k)}$, the nearest point $p(P,E)$ lies in the relative interior of a uniquely determined face $F$ of $P$. For a given face $F$, a Borel set $C$ in $N(d,k)$ and $\epsilon >0$, we compute the measure of the set
$$
A := \{ E\in M^{(k)}_\epsilon(P,C) : p(P,E)\in \relint\, F \} .
$$
If $\dim F \ge d-k$, the nearest point in $F$ to a given $k$-flat $E$ is either not unique or not in the relative interior of $F$ (or $E$ and $F$ intersect). In each of these cases, $E\notin A$, hence $A$ is empty. Therefore, we can concentrate on the case 
$$m:= \dim F \le d-k-1.$$ 
Since $\mu_k$ is translation invariant, we may also assume that $0\in F$ (of course, $C$ also has to be replaced by a corresponding translate). 
Then
\begin{align*}
\mu_k(A) &= \int_{A(d,k)} {\bf 1}\{ E\in P_\epsilon^{(k)}, p(P,E)\in \relint(F), \pi (E)\in C\} \,\mu_k (dE)\cr
& = \int_{G(d,k)} I(L)\,\nu_k (dL),
\end{align*}
where 
$$
I(L):=\int_{L^\bot} {\bf 1}\{ L+y\in P_\epsilon^{(k)},  p(P,L+y)\in \relint(F), \pi (L+y)\in C\}\, 
 \HH^{d-k}(dy).
$$

Next we investigate  $I(L)$.

If $L$ and $F$ are not in general relative position and $p(P,L+y)\in \relint(F)$ holds, then $L+y\notin P^{(k)}$ and thus $I(L)=0$. We therefore can concentrate on those subspaces $L$ which are in general relative position with respect to $F$. In any case, since we integrate 
$L$ over $G(d,k)$ with respect to the Haar measure $\nu_k$, we can always assume that $F$ and $L$ are in general relative position, excluding a subset of $G(d,k)$ of measure zero \cite[Lemma 13.2.1]{SW}. Then, let
$$
L_1:= L(F)|L^\perp = (L(F)+L)\cap L^\perp
$$
and
$$
L_2:= L_1^\perp\cap L^\perp = L(F)^\perp\cap L^\perp .
$$
It follows that $L_1\perp L_2$ and $L^\perp = L_1 \oplus L_2.$ Moreover, for $y_1\in L_1$ and $y_2\in L_2$, we have
$$
(L+y_1+y_2)|L(F) = (L+y_1)|L(F),\quad (L+y_1)|L^\perp = \{y_1\},
$$
and 
$$
u(F,L+y_1+y_2) = \frac{y_2}{\| y_2\|} , \quad d(F,L+y_1+y_2 ) = \| y_2\| ,
$$
whenever $p(F, L+y_1+y_2)\in\relint(F)$. Thus, $p(P,L+y_1+y_2)\in\relint(F)$ is equivalent to $y_1\in \relint(F) |L^\perp, y_2\in N(P,F)$ and, in this case, $0< d(F,L+y_1+y_2)\le\epsilon$ is equivalent to $0< \| y_2\|\le\epsilon$.

We obtain
\begin{align*}
I(L) =& \int_{L_1}\int_{L_2} {\bf 1}\{ L+y_1+y_2\in P_\epsilon^{(k)}, p(P,L+y_1+y_2)\in \relint(F)\}\cr
&\ \times {\bf 1}\{ \pi (L+y_1+y_2)\in C\} \,\HH^{d-k-m} (dy_2)\,\HH^{m}(dy_1)\cr
=& \int_{L_1}\int_{L_2} {\bf 1}\{ L+y_1+y_2\in P^{(k)}, y_1\in \relint(F) |L^\perp, 0<\| y_2\|\le \epsilon\}\cr
&\ \times {\bf 1}\{y_2\in N(P,F), (p(F,L+y_1),y_2/\| y_2\| ,L)\in C\}\,\HH^{d-k-m} (dy_2)\,\HH^{m} (dy_1) .
\end{align*}
If $y_1\in \relint(F)|L^\perp$ and $y_2\in N(P,F)\setminus\{0\}$, then  $L+y_1+y_2\in P^{(k)}$, since $L$ and $F$ are in general relative position.  Hence,
$$
I(L)
= \frac{\epsilon^{d-k-m}}{d-k-m} \int_{F|L^\perp}\int_{L^\perp\cap n(P,F)} {\bf 1}\{ (p(F,L+y),u,L)\in C\}\, \HH^{d-k-m-1} (du)
\,\HH^{m} (dy) ,
$$
which gives us
\begin{align*}
\mu_k(A)
=& \frac{\epsilon^{d-k-m}}{d-k-m}\int_{G(d,k)}\int_{F|L^\perp}\int_{L^\perp\cap n(P,F)} \cr
&\times {\bf 1}\{ (p(F,L+y),u,L)\in C\} \,\HH^{d-k-m-1} (du)\,\HH^{m} (dy)\,\nu_k(dL) .
\end{align*}
The last expression is translation covariant, so we need no longer assume $0\in F$. 
Summing over all $F\in{\cal F}_m(P)$ and all $m=0,\ldots,d-k-1$, we arrive at
\begin{align*}
\mu_\epsilon^{(k)}(P,C)
=& \sum_{m=0}^{d-k-1}\frac{\epsilon^{d-k-m}}{d-k-m}\sum_{F\in {\cal F}_m(P)}\int_{G(d,k)}\int_{F|L^\perp}\int_{L^\perp\cap n(P,F)} \cr
&\times {\bf 1}\{ (p(F,L+y),u,L)\in C\} \,\HH^{d-k-m-1} (du)\,\HH^{m} (dy)\,\nu_k(dL)\cr
 = &\frac{1}{d-k}\sum_{m=0}^{d-k-1}\epsilon^{d-k-m}\binom{d-k}{ m} \Theta^{(k)}_m(P,C) , 
\end{align*}
where 
\begin{align}
\Theta^{(k)}_m(P,C)
:=& \binom{d-k-1}{m}^{-1}\sum_{F\in {\cal F}_m(P)}\int_{G(d,k)}\int_{F|L^\perp}\int_{L^\perp\cap n(P,F)} \label{gensuppmeas}\\
&\times {\bf 1}\{ (p(F,L+y),u,L)\in C\}\, \HH^{d-k-m-1} (du)\,\HH^{m} (dy)\,\nu_k(dL) ,\notag
\end{align}
for $m=0,\ldots,d-k-1$. 
\end{proof}

\begin{lemma}\label{l3} There exist real numbers $a_1(d,k,m),\ldots,a_{d-k}(d,k,m)$, depending only on $d,k$ and $m$, such that
\begin{equation}\label{localsteiner3}
\Theta^{(k)}_m(P,\cdot) =  \sum_{i=1}^{d-k}a_i(d,k,m)\mu^{(k)}_{i}(P,\cdot) ,
\end{equation}
for $m=0,\ldots,d-k-1$ and all polytopes $P$.
\end{lemma}

\begin{proof}
For a polytope $P$ and a Borel set $C\subset N(d,k)$, Lemma \ref{l2} proves the existence of measures $\Theta^{(k)}_0(P,\cdot), \ldots, \Theta^{(k)}_{d-k-1}(P,\cdot)$ such that
\begin{equation*}
\mu^{(k)}_{\epsilon}(P,C) =\sum_{m=0}^{d-k-1} \epsilon^{d-k-m}\left[\frac{1}{d-k}\binom{d-k}{ m} \Theta^{(k)}_m(P,C)\right] ,
\end{equation*}
for all $\epsilon >0$. Choosing $\epsilon = 1,2,\ldots,d-k$, we obtain a system of linear equations for the unknowns
$$
\left[\frac{1}{d-k}\binom{d-k}{ m} \Theta^{(k)}_m(P,C )\right] ,\quad m=0,\ldots,d-k-1,$$
where the coefficient matrix is invertible (the determinant is a Vandermonde determinant) and the entries only depend on $d,k$ and $m$.
Solving the system, we obtain the assertion.
\end{proof}

\begin{proof}[Proof of Theorem \ref{steiner}] (a)  
 We fix $k$. Let $K$ be a convex body. Then we define
 \begin{equation}\label{inversion}
\Theta^{(k)}_m(K,\cdot) :=  \sum_{i=1}^{d-k}a_i(d,k,m)\mu^{(k)}_{i}(K,\cdot) ,\quad m=0,\ldots,d-k-1.
\end{equation}
Clearly, these measures are finite and weakly continuous as functions of $K$, since the measures on the right-hand side have these properties. Lemma \ref{L21} implies that the measures on the right-hand side are concentrated on $\Nor_k(K)$, hence this is also true for the measure on  the left-hand side. 
Moreover, the definition is consistent with the case of polytopes (by Lemma \ref{l3}). Since $\Theta^{(k)}_m(K,\cdot)$ is the 
weak limit of a sequence $\Theta^{(k)}_m(P_n,\cdot)$, $n\in\N$, for a sequence of polytopes $P_n$ with $P_n\to K$, these measures are  non-negative. 
Moreover,  using weak convergence in \eqref{localsteiner2}, we  see that
$$
\mu^{(k)}_{\epsilon}(K,\cdot) = \frac{1}{d-k}\sum_{m=0}^{d-k-1}\epsilon^{d-k-m}\binom{d-k}{ m} \Theta^{(k)}_m(K,\cdot) 
$$
holds for arbitrary $K$ and $\epsilon >0$. 

(b) The fact that $K\mapsto \Theta^{(k)}_m(K,\cdot)$ is weakly continuous has already been shown in (a).  From \eqref{inversion} and Lemma \ref{l1}  we obtain that $\Theta^{(k)}_m(K,\cdot)$ depends additively on $K$. 

(c) This follows from (b) with the help of Lemma 12.1.1 in \cite{SW} (alternatively, from Lemma \ref{l1}). 
\end{proof}

\medskip

The following polynomial expansion of $\Theta^{(k)}_m(K+\epsilon B^d,\cdot)$ follows easily from the defining equation \eqref{localsteiner} and is included here for the sake of completeness. Two different proofs can be found in \cite{Hind} and \cite{HTW}.
 
\begin{theorem}\label{generalsteiner}
Let $k\in\{0,\ldots, d-1\}$, $K\in{\cal K}$, let $C\subset N(d,k)$ be a Borel set, and let $m\in\{0,\ldots, d-k-1\}$. Then
$$
\Theta^{(k)}_m(K+\epsilon B^d,t_\epsilon (C)) = \sum_{j=0}^m \epsilon^j \binom{m}{ j} \Theta^{(k)}_{m-j}(K,C ),
$$
where $t_\epsilon : N(d,k) \to N(d,k), (x,u,L)\mapsto (x+\epsilon u, u, L)$.
\end{theorem}

We mention some measures obtained as projections. The Borel measure $S^{(k)}_{m}(K,\cdot )$ on   $F^\perp (d,k)$, which 
is defined by
$$
S^{(k)}_{m}(K,\cdot) :=  \Theta^{(k)}_{m}(K,\R^d\times \cdot ),
$$
is called the $m$-th {\it $k$-flag area measure} of $K$, for $m\in\{0,\ldots,d-k-1\}$. Notice that the flag manifold $F^\perp (d,k)$ is isomorphic to   $\Nor_k (B^d)$. The measure $S^{(0)}_{m}(K,\cdot )$ is (up to a natural identification) the ordinary $m$-th area measure  $S_{m}(K,\cdot )$ of $K$. This follows easily by comparing the polynomial expansion \eqref{localsteiner}, for $k=0$ and $C = \R^d\times \cdot$, with the results in \cite[Chapter 4]{S}. The other projection is a Borel measure  on $\R^d\times G(d,k)$, which is defined by 
$$
C^{(k)}_{m}(K,\cdot ) :=  \Theta^{(k)}_{m}(K,\{ (x,u,L)\in \R^d\times S^{d-1}\times G(d,k) : (x,L)\in \cdot\} ).
$$
We call it the $m$-th {\it $k$-flag curvature measure} of $K$, for $m\in\{0,\ldots,d-k-1\}$. 
The measure $C^{(0)}_{m}(K,\cdot )$ is the ordinary $m$-th curvature measure  $C_{m}(K,\cdot )$ of $K$ (see again \cite[Chapter 4]{S})  and  $\Theta^{(0)}_{m}(K,\cdot)$ is the ordinary $m$-th support measure of $K$, for $m\in\{0,\ldots,d-1\}$. More generally, from \eqref{stern} and \cite[Proposition 1]{HTW} we conclude that
$$
\Theta^{(k)}_m(K,\cdot\times G(d,k))=\frac{\omega_{d-k}}{\omega_d}\, \Theta_m(K,\cdot),
$$
for $m\in\{0,\ldots,d-1\}$. In particular, we have
\begin{equation}\label{specialS}
S^{(k)}_m(K,\cdot\times G(d,k))=\frac{\omega_{d-k}}{\omega_d}\, S_m(K,\cdot).
\end{equation}

In order to clarify the connection between the flag measures introduced by Theorem \ref{steiner} and the measure defined by \eqref{secondflagmeasure}, we observe that for a polytope $P$ and a Borel set $A\subset F^\perp (d,k)$, a special case of 
 \eqref{gensuppmeas} implies that
\begin{align}\label{polsurface} 
S^{(k)}_m(P,A)
=& \binom{d-k-1}{m}^{-1}\sum_{F\in {\cal F}_m(P)}\int_{G(d,k)}V_m(F|L^\perp)\int_{L^\perp\cap n(P,F)} \cr
&\times {\bf 1}\{ (u,L)\in A\} \,\HH^{d-k-m-1} (du)\,\nu_k(dL) ,
\end{align}
where $m\in\{0,\ldots,d-k-1\}$. The next theorem provides an alternative representation of this measure. 

\bigskip

\begin{theorem}\label{altrepr}
Let $P\in{\cal P}$, $k\in\{0,\ldots,d-1\}$ and $m\in\{0,\ldots, d-k-1\}$. Let $f:F^\perp(d,k)\to[0,\infty)$ be measurable. Then
\begin{align}\label{altrepr-a}
&\binom{d-k-1}{m}\int_{F^\perp (d,k)} f(u,L) \,S^{(k)}_m(P,d(u,L)) \\
&\qquad = \frac{\omega_{d-k}}{\omega_d} \sum_{F\in {\cal F}_m(P)} V_m(F) 
\int_{n(P,F)}\int_{G(u^\perp,k)}
 [F,L]^2  f(u,L)\,\nu^{u^\perp}_k(dL)\, \HH^{d-m-1} (du).\notag 
\end{align}
\end{theorem}

\begin{proof} From \eqref{polsurface} and passing to orthogonal subspaces, we obtain 
\begin{align*}
&\binom{d-k-1}{m}\int_{F^\perp (d,k)} f(u,L) \,S^{(k)}_m(P,d(u,L))\\
&\quad = \sum_{F\in {\cal F}_m(P)}\int_{G(d,d-k)}V_m(F|U)\int_{U\cap n(P,F)} 
 f(u,U^\perp) \,\HH^{d-k-m-1} (du)\,\nu_{d-k}(dU) \\
&\quad = \sum_{F\in {\cal F}_m(P)}V_m(F)\int_{G(d,d-k)}\int_{U\cap n(P,F)} 
[F,U^\perp]  f(u,U^\perp)\,\HH^{d-k-m-1} (du)\,\nu_{d-k}(dU) .
\end{align*}
Next we interchange the order of integration by applying \cite[Theorem 1]{AZ91}. 
Specifically, we replace 
$m$ in \cite[Theorem 1]{AZ91} by $d-m-1$, $p$ by $d-k$, and choose $k$ there as $1$. Moreover, we consider 
the $\mathcal{H}^{d-m-1}$-rectifiable set $n(P,F)\subset S^{d-1}$ and the function $h(u,U)=f(u,U^\perp)[F,U^\perp]$, 
independent of the tangent space $T_u n(P,F)$. The expression $J(T_u n(P,F),U)$ in \cite[Theorem 1]{AZ91} is introduced 
in \cite[p.~336, (2)]{AZ91} and defined as the $k$-volume of the orthogonal projection of a $k$-dimensional unit cube in 
$V:=\left(T_u n(P,F)\cap U\right)^\perp\cap T_u n(P,F)$ to $U^\perp$. Since 
$$V=\left(\left(L(F)\oplus\langle u\rangle\right)^\perp\cap (U\cap u^\perp)\right)^\perp\cap (L(F)\oplus\langle u\rangle)^\perp,$$ 
it follows that 
\begin{align*}
J(T_u n(P,F),U)&=[V^\perp,U^\perp]=[\left(L(F)\oplus\langle u\rangle\right)^\perp\cap(U\cap u^\perp)+L(F)\oplus\langle u\rangle,U^\perp]\\
&=[L(F)\oplus\langle u\rangle,U^\perp]=[F,U^\perp].
\end{align*}
Moreover, the constant $\beta(d,p,k)$, which  
is defined in \cite[p.~336-7]{AZ91}, is a ratio of Hausdorff measures of two sets. This constant simplifies considerably in the case of $\beta(d,d-k,1)$, for which we obtain
$$
\beta(d,d-k,1)=\Gamma\left(\frac{1}{2}\right)^{-k}\frac{\Gamma\left(\frac{d}{2}\right)}{\Gamma\left(\frac{d-k}{2}\right)}
=\frac{\omega_{d-k}}{\omega_d}.
$$
Finally, recall that $\langle u\rangle$ is the linear subspace spanned by $u$. Then we get\begin{align*}
&\binom{d-k-1}{m}\int_{F^\perp (d,k)} f(u,L) \,S^{(k)}_m(P,d(u,L))\\
&\quad = \frac{\omega_{d-k}}{\omega_d}\sum_{F\in {\cal F}_m(P)}V_m(F)\int_{n(P,F)}
\int_{G(\langle u\rangle,d-k)}[F,U^\perp]^2
 f(u,U^\perp) \,\nu_{d-k}^{\langle u\rangle}(dU)\,\HH^{d-m-1}(du) \\
&\quad = \frac{\omega_{d-k}}{\omega_d}\sum_{F\in {\cal F}_m(P)}V_m(F)\int_{n(P,F)}
\int_{G( u^\perp,k)}[F,L]^2
 f(u,L) \,\nu_{k}^{u^\perp}(dL)\,\HH^{d-m-1}(du),
\end{align*}
which proves the assertion.
\end{proof}

\bigskip

\begin{rem} {\rm In \cite{BHW} we provide a measure geometric approach to flag support measures, which leads to a representation of $\Theta^{(k)}_m(K,\cdot)$, 
and therefore also of $S^{(k)}_m(K,\cdot)$, for a general convex body $K$. In the case of 
polytopes, the representation for $S^{(k)}_m(K,\cdot)$ obtained in \cite{BHW} reduces to Theorem \ref{altrepr}.}
\end{rem}

\begin{rem} {\rm From \eqref{altrepr-a} we deduce that
\begin{align*} 
&\binom{d-k-1}{m}\,S^{(k)}_m(P,\cdot\times G(d,k)) \\
& \quad = \frac{\omega_{d-k}}{\omega_d} \sum_{F\in {\cal F}_m(P)} V_m(F) 
\int_{n(P,F)}   \mathbf{1}\{u\in\cdot\} \int_{G(u^\perp,k)}
 [F,L]^2 \,\nu^{u^\perp}_k(dL)\, \HH^{d-m-1} (du).
\end{align*}
Since \cite[p.~139 and 3.2.13, p.~251]{Federer1969} imply that, independently of $u\in S^{d-1}$ and $L(F)\subset u^\perp$, 
$$
\int_{G(u^\perp,k)}
 [F,L]^2 \,\nu^{u^\perp}_k(dL)=\left(\frac{\beta_2(d-1,m)}{\beta_2(d-1-k,m)}\right)^2=\frac{\binom{d-1-k}{m}}{\binom{d-1}{m}},
$$
we obtain again \eqref{specialS}.
}
\end{rem}

\bigskip

Next we consider the special case $m=j$ and $k=d-1-j$ of Theorem  \ref{altrepr}. In particular, for a measurable function 
$f:F(d,d-j)\to [0,\infty)$ we get

\begin{align*}
& \int_{F^\perp (d,d-1-j)} f(u,L\oplus \langle u\rangle) \,S^{(d-1-j)}_j(P,d(u,L))\\
&\quad =\frac{\omega_{j+1}}{\omega_d}\sum_{F\in {\cal F}_j(P)}V_j(F)\int_{n(P,F)}
\int_{G( u^\perp,d-1-j)}[F,L]^2
 f(u,L\oplus \langle u\rangle) \\
 &\qquad \qquad \,\nu_{d-1-j}^{u^\perp}(dL)\,\HH^{d-j-1}(du)\\
&\quad =\frac{\omega_{j+1}}{\omega_d}\sum_{F\in {\cal F}_j(P)}V_j(F)\int_{n(P,F)}
\int_{G( \langle u\rangle ,d-j)}[F,U]^2
 f(u,U) \\
&\qquad \qquad \ \nu_{d-j}^{\langle u\rangle}(dU)\,\HH^{d-j-1}(du)\\
&\quad =\frac{\omega_{j+1}}{\omega_d}\int_{F(d,d-j)} f(u,U)\, \psi_j(P,d(u,U)),
\end{align*}
where we used that $[F,L]=[F,L\oplus \langle u\rangle ]$ if $L(F),L\subset u^\perp$. 

Since $K\mapsto S^{(d-1-j)}_j(K,\cdot)$ is weakly continuous, 
the following is a direct consequence.

\begin{koro}\label{corol2} For $j=0,\ldots, d-1$, the measure $\psi_j(K,\cdot)$, 
defined by \eqref{secondflagmeasure} for polytopes $K$, has a continuous extension to all convex bodies $K$. 
\end{koro}

\section{Valuations with continuous extensions}

We now show that strongly flag-continuous valuations on polytopes, which are homogeneous 
of degree $j\in\{1,\ldots,d-2\}$, have a continuous extension to all convex bodies.

\begin{theorem}\label{ext}
Let $\varphi = \varphi_j$ be a $j$-homogeneous valuation on ${\cal P}$, $j\in\{1,\ldots,d-2\}$, which is strongly flag-continuous. Then, $\varphi$ has a (unique) extension to ${\cal K}$ which is continuous in the Hausdorff metric.
\end{theorem}

\begin{proof} For a polytope $P$, we have
$$
\varphi (P) =  \sum_{F\in {\cal F}_j(P)} f_{j}(n(P,F)) V_j(F).
$$
Since $\varphi$ is strongly flag-continuous,
\begin{equation*}
f_j (p) = \int_p \tilde f_j(u,\langle p\rangle) \,\HH^{d-j-1}(du),
\end{equation*}
if $ p\in {\wp}_{d-j-1}$ has dimension $d-1-j$, 
and
\begin{equation*}
\tilde f_j (u,L) = \int_{G(\langle u\rangle ,d-j)} [M,L^\perp]^2 \tilde g_j(u,M)
 \,\nu^{\langle u\rangle}_{d-j}(dM) ,\quad (u,L)\in F(d,d-j),
\end{equation*}
for some continuous function $\tilde g_j$ on $F(d,d-j)$.

Hence,
\begin{align*}
\varphi (P)
&\ =  \sum_{F\in {\cal F}_j(P)}  V_j(F)\int_{n(P,F)} \int_{G(\langle u\rangle ,d-j)} [F,M]^2\tilde g_j(u,M)
\, \nu^{\langle u\rangle}_{d-j}(dM)\, \HH^{d-j-1}(du)\cr
&\ = \int_{F(d,d-j)} \tilde g_j(M,u)\, \psi_j (P, d(M,u))
\end{align*}
by \eqref{secondflagmeasure}.

We now define,
$$\varphi (K) = \int_{F(d,d-j)} \tilde g_j(M,u)\, \psi_j (K, d(M,u)),
$$
for $K\in {\cal K}$, where $\psi_j (K,\cdot )$ is the extension ensured by Corollary \ref{corol2}. Since $K\mapsto \psi_j (K,\cdot )$ is  weakly continuous, the functional $\varphi$ is continuous on ${\cal K}$ (and  a translation invariant valuation).
\end{proof}

\bigskip

\noindent
Authors' addresses: 

\bigskip

\noindent
Wolfram Hinderer, 
Robert-Koch-Str. 196, 
D-73760 Ostfildern, 
wolfram@hinderer-lang.de

\bigskip

\noindent
Daniel Hug, Karlsruhe Institute of Technology (KIT), 
Department of Mathematics, 
D-76128 Karlsruhe, daniel.hug@kit.edu

\bigskip

\noindent
Wolfgang Weil, Karlsruhe Institute of Technology (KIT), 
Department of Mathematics, 
D-76128 Karls\-ruhe, wolfgang.weil@kit.edu

\end{document}